\documentclass{amsart}
\usepackage{amssymb}
\usepackage{amsmath}
\usepackage{graphicx}
\usepackage{enumerate}
\usepackage{color}
\newcommand{\G}{\mathcal{G}}
\newcommand{\Z}{\mathbb{Z}}
\newcommand{\R}{\mathbb{R}}

\newcommand{\N}{\mathbb{N}}

\newcommand{\eps}{\varepsilon}
\newcommand{\seq}[1]{\langle #1 \rangle}

\newtheorem{thm}{Theorem}[section]
\newtheorem{cor}[thm]{Corollary}
\newtheorem{rem}[thm]{Remark}
\newtheorem{que}[thm]{Question}
\newtheorem{lem}[thm]{Lemma}
\newtheorem{defn}[thm]{Definition}
\newtheorem{prop}[thm]{Proposition}

\newtheorem*{con}{Conjecture}
\newtheorem*{thm-A}{Theorem~\ref{thm-A}}
\newtheorem*{thm-B}{Theorem~\ref{thm-B}}
\newtheorem*{thm-C}{Theorem~\ref{thm-C}}
\newtheorem*{cor-D}{Corollary~\ref{cor-D}}
\newtheorem*{thm-E}{Theorem~\ref{thm-E}}

\newcommand{\Orb}{\mathrm{Orb}}

\DeclareMathOperator{\Int}{int}

\DeclareMathOperator{\dist}{dist}
\DeclareMathOperator{\diam}{diam}

\begin{document}
%

\title{Edrei's Conjecture revisited
}

\author{Jan P. Boro\'nski}
\address[J. P. Boro\'nski]{AGH University of Science and Technology, Faculty of Applied
	Mathematics, al.
	Mickiewicza 30, 30-059 Krak\'ow, Poland
	-- and --
	National Supercomputing Centre IT4Innovations, Division of the University of Ostrava,
	Institute for Research and Applications of Fuzzy Modeling,
	30. dubna 22, 70103 Ostrava,
	Czech Republic}
\email{jan.boronski@osu.cz}
\author{Ji\v{r}\'{\i} Kupka}
\address[J. Kupka]{National Supercomputing Centre IT4Innovations, Division of the University of Ostrava,
	Institute for Research and Applications of Fuzzy Modeling,
	30. dubna 22, 70103 Ostrava,
	Czech Republic}
	\email{jiri.kupka@osu.cz}
\author{Piotr Oprocha}
\address[P. Oprocha]{AGH University of Science and Technology, Faculty of Applied
	Mathematics, al.
	Mickiewicza 30, 30-059 Krak\'ow, Poland
	-- and --
	National Supercomputing Centre IT4Innovations, Division of the University of Ostrava,
	Institute for Research and Applications of Fuzzy Modeling,
	30. dubna 22, 70103 Ostrava,
	Czech Republic}
\email{oprocha@agh.edu.pl}
\maketitle
\begin{abstract}
Motivated by a recent result of Ciesielski and Jasi\'nski we study periodic point free Cantor systems that are conjugate to systems with vanishing derivative everywhere, and more generally locally radially shrinking maps. Our study uncovers a whole spectrum of dynamical behaviors attainable for such systems, providing new counterexamples to the Conjecture of Edrei from 1952, first disproved by Williams in 1954.
\end{abstract}
\section{Introduction}
The present paper is concerned with the following question:
\begin{que}
What Cantor set homeomorphisms are conjugate to homeomorphisms with vanishing derivative everywhere?
\end{que}
The motivation for this question comes from the fixed point theory of contractive and locally contractive mappings. The celebrated Banach Fixed Point Theorem \cite{Banach} asserts that every contraction on a complete metric space has a unique fixed point. Recall that for a complete metric space $(X,d)$ we call a map $f:X\to X$:
\begin{itemize}
\item \textit{contraction} if there exists an $L<1$ such that $d(f(x),f(y))\leq Ld(x,y)$ for all $x,y\in X$;
\item \textit{local contraction} if for every $x\in X$ there exists an $L_x<1$ and $q_x>0$ such that $d(x,y)<q_x$ and $d(x,z)<q_x$ implies $d(f(y),f(z))\leq L_xd(y,z)$;
\item \textit{weak local contraction} if for every $x\in X$ there exists an $r_x>0$ such that $d(x,y)<r_x$ implies $d(f(x),f(y))\leq d(x,y)$;
\item \textit{local isometry} if for every $x\in X$ there exists an $R_x>0$ such that $d(x,y)<R_x$ implies $d(f(x),f(y))=d(x,y)$.
\end{itemize}
In 1961 Edelstein generalized Banach's result to the local setting \cite{Edelstein2}, \cite{Edelstein}, proving that for every local contraction $f$ on a compact space $X$ there exists an integer $n$ such that $f^n$ has a fixed point (see Theorem 6 in \cite{Ciesielski}). For weak local contractions Edelstein's results does not apply and earlier, in 1952, Edrei stated the following conjecture.
\begin{con}[Edrei, \cite{Edrei}]
Suppose $X$ is a compact metric space and $f:X\to X$ is a weak local contraction. Then $f$ is a local isometry.
\end{con}
Edrei's conjecture was disproved in 1954 by Williams \cite{Williams} who constructed four different examples of maps for which every point is a weak contraction point, but for which there exist points that are not isometry points. His first example was a map on a planar compactum $M$ with a fixed point, at which the map was noninjective and it was a local isometry at all points but one. The second example had a single fixed point, a single point that was not an isometry point, and similar to the first example contained isolated points. By extending his first example linearly to the cone over $M$ he then obtained a counterexample on a 1-dimensional continuum, with a homeomorphism that possessed a nonempty set of fixed points, and a circle of local isometry points. His last example was an extension of a minimal isometry on a Cantor set $C$, to a map defined on the union $M'$ of $C$ with a sequence of points converging to $C$, and thus again contained isolated points. It is then natural to ask if these counterexamples must always have either isolated points or fixed points, and if there exists a weak local contraction that is a local isometry on no subset. This kind of a map is called \textit{locally radially shrinking} map; i.e. a map $f:X\to X$ such that
\begin{itemize}
\item[(LRS)] for every $x\in X$ there exists an $\epsilon_x>0$ such that $d(x,y)<\epsilon_x$ implies $d(f(x),f(y))<d(x,y)$ for all $y\neq x$.
\end{itemize}Clearly none of the aforementioned maps constructed by Williams is (LRS). Any map whose derivative vanish on a set $K\subseteq \mathbb{R}$ is locally radially shrinking on $K$. Since this class contains all constant maps, for such maps it does not only seem unlikely to be a homeomorphism, but also given Edelstein's result one may expect a fixed, or periodic point. Surprisingly, a minimal locally radially shrinking Cantor set homeomorphism $\mathfrak{f}$ has been discovered recently by Ciesielski and Jasi\'nski, who embedded a 2-adic odometer into the real line with vanishing derivative everywhere. They also proved that if $X$ is an infinite compact metric space and $f:X\to X$ is onto and has the (LRS) property then there exists a perfect subset $Y$ such that $f|Y$ is minimal. Note that Bruckner and Steele \cite{Bruckner} proved that most sets cannot be mapped by a Lipschitz function over a given Cantor set; i.e. given a Cantor set $E\subseteq [0,1]$ the collection of all closed subsets $F$ of $[0,1]$ for which there exists a Lipschitz function $f$ with $E\subseteq f(F)$ is of first category in the hyperspace of all compact subsets of $[0,1]$. Ciesielski and Jasi\'nski noted in \cite{Ciesielski} that \textit{the homeomorphism $\mathfrak{f}$ might mark the spot where the minimal dynamical systems `meet' Banach Fixed-Point Theorem}. It is then of interest to determine how fine is the line separating the two phenomena. Is the homeomorphism $\mathfrak{f}$ just an isolated example, or is there a wider class of dynamical systems that fall into the same category? In the present paper we address this question in the following way. First we show that the result of Ciesielski and Jasi\'nski generalizes to all odometers\footnote{Theorem \ref{thm-A} answers a question of Emma d'Aniello raised at the Thirty-Ninth Summer Symposium in Real Analysis held at St. Olaf College in Northfield, MN, June 8 - 15, 2015, who asked if every odometer is conjugate to a homeomorphism with vanishing derivative everywhere. We are grateful to K.C. Ciesielski for bringing this fact to our attention.}.
\begin{thm-A}
Every odometer is conjugate to a homeomorphism $\mathfrak{f}:C\to C$ such that $\mathfrak{f}'\equiv 0$ and $\mathfrak{f}$ extends to a differentiable surjection $\bar{\mathfrak{f}}:\mathbb{R}\to\mathbb{R}$.
\end{thm-A}
We employ, however, a different approach which, in the special case of the 2-adic adding machine, provides a substantially shorter proof of their original result. We then go on to demonstrate other systems that are conjugate to systems with zero derivative. The first one shows that such systems do not need to be equicontinuous (in fact may be not equicontinuous at any point).
\begin{thm-B}
There exists a minimal weakly mixing Cantor set homeomorphism $T:X\to X$ that embeds in $\mathbb{R}$ with vanishing derivative everywhere.
\end{thm-B}
Next we show that minimality is not a necessary property for periodic point free systems in this class.
\begin{thm-C}
There exists a transitive, nonminimal and periodic point free Cantor set homeomorphism that embeds in $\mathbb{R}$ with vanishing derivative everywhere.
\end{thm-C}
In fact any Cantor set minimal homeomorphism with (LRS) can be extended to a transitive nonminimal homeomorphism with (LRS) of a new space $Z$, such that $Z$ contains isolated points and the homeomorphism exhibits attractor-repellor dynamics (see Theorem \ref{Th:attr:rep}). Note that in this case the set of isolated points does not allow one to speak of a derivative, and this is where the (LRS) property is very natural to investigate instead. This also comes handy in the following example, where the Cantor set is embedded in $\mathbb{R}^2$.
\begin{cor-D}
	There exists a Cantor set $C\subseteq\mathbb{R}^2$ and a homeomorphism $F$ such that $F$ has the (LRS) property, $C=\bigcup_{i\in I} M_i$ where $I$ is uncountable, $M_i\cap M_j=\emptyset$ for $i\neq j$
	and $(M_i,F)$ is minimal for every $i$.
\end{cor-D}
Finally we show that in this class there exists a nontransitive homeomorphism on the Cantor set $W\subseteq \mathbb{R}^2$ with a single fixed point, and no other periodic points.
It shows that homeomorphisms on Cantor set can have (LRS) and a fixed point, while it is not obvious from the definition that such a system can exist.
\begin{thm-E}
There exists a Cantor set $W\subseteq\mathbb{R}^2$ and a nontransitive homeomorphism $G$ with the (LRS) property such that the set of periodic points of $G$ consists of a single fixed point.
\end{thm-E}
Note that by piecing together several disjoint copies of the above dynamical system, it is easy to give examples with periodic orbits of other periods. However, the following conjecture remains open.
\begin{con}
For every minimal dynamical systems $(C,T)$ on a Cantor set $C$ there exists an equivalent metric $(C,\rho)$ such that $T$ has (LRS) property with respect to $\rho$.
\end{con}

\section{Preliminaries}
A compact metric space $C$ is \textit{Cantor set} if it is totally disconnected and does not have isolated points. It is known that
all Cantor sets are homeomorphic. Furthermore, if $X\subset \R$ is a perfect set then standard definition of derivative makes sense for any map $f\colon X\to X$.
One of the nicest applications of such generalization of the concept of derivative is the following Jarn\'{\i}k theorem (see \cite{Jar}, cf. \cite{Ciesielscy, Koc,Nicol})

\begin{thm}[Jarn\'{\i}k]\label{thm:Jarnik}
	Let $X\subset \R$ be a perfect set and let $f\colon X\to \R$ be differentiable. Then there exists a differentiable extension $F\colon \R \to \R$ of $f$, that is $F|_X=f$.
\end{thm}

\begin{rem}
It is clear that if $f'\equiv 0$ then $f$ has (LRS).
\end{rem}

\subsection{A remark on global shrinking}
As a matter of introduction, and to highlight contrast to the main results on locally radially shrinking maps, we recall the following two facts that pertain to globally shrinking maps on compact spaces.
\begin{defn}
	Let $(X,d)$ be a compact metric space. A map $f\colon X\to X$ is called \textit{shrinking} if
	$d(f(x),f(y))<d(x,y)$ for any $x\neq y$.
\end{defn}
\begin{prop}\label{lem:1}
	If $f:X\to X$ is a surjective shrinking map then $X$ is a single point.
\end{prop}
\begin{prop}
	If $f:X\to X$ is a shrinking map then:
	\begin{enumerate}
		\item there exists a unique fixed point $z=f(z)$ (see \cite{Edelstein}),
		\item for every $x\neq z$ there exists $n$ such that $f^{-n}(\{x\})=\emptyset$.
	\end{enumerate}
\end{prop}
\noindent The proofs are easy exercises.

\subsection{Topological dynamics}
Let $(X,d)$ be a compact metric space and $f:X\to X$ be continuous. Then a pair $(X,f)$ is a \emph{discrete dynamical system}. A \emph{Cantor set} is any 0-dimensional compact metric space without isolated points, that is every space homeomorphic with the standard Cantor set.

A subset $A\subseteq X$ is \emph{$f$-invariant} if $f(A)= A$. A dynamical system $(X,f)$ is
\begin{enumerate}[(i)]
  \item \emph{minimal} if it does not contain any nonempty proper $f$-invariant closed subset;
  \item \emph{transitive} if for any two nonempty open subsets $U,V\subseteq X$ there exists $n\in \N$ for which $f^n (U) \cap V \neq \emptyset$;
  \item \emph{weak mixing} if the product system $(X\times X, f\times f)$ is transitive;
  \item \emph{equicontinuous} if for every $\eps>0$ there is $\delta>0$ such that if $d(x,y)<\delta$
	then $d(f^n(x),f^n(y))<\eps$ for every $n\geq 0$;
  \item \emph{sensitive} if there exists $\delta >0$ such that for any $x\in X$ and $\eps >0$ there exist $y\in X$, $d(x,y) <\eps$, and $n\in \N$ such that $d(f^n(x),f^n(y))>\delta$.
\end{enumerate}

For fixed $\eps >0$ and $n\in\N$, a set $A\subseteq X$ is $(n,\eps)$-separated if for $x,y\in A$, $x\neq y$, we have $d(f^j (x), f^j (y))>\eps$ for some $j\in\{0,1, \ldots ,n-1\}$. Let $s(n,\eps)$ be the maximal cardinality of $(n,\eps)$-separated set in $X$. Then \emph{topological entropy} $h(f)$ of the map $f$ is defined by
$$
h(f) = \lim_{\eps \to 0} \limsup_{n\to \infty} \frac {\log s(n,\eps)}{n}.
$$

\subsection{Graph covers}

By a \textit{graph} we mean a pair $G=(V,E)$ of finite sets, where elements of $V$ represent \emph{vertices} and elements of $E\subseteq V\times V$ represent \emph{edges} of the graph $G$. The graph $G$ is \emph{edge surjective} if every vertex has incoming and outgoing edge, i.e. for every $v\in V$ there are $u,w\in V$ for which $(u,v),(v,w) \in E$. For graphs $(V_1,E_1)$, $(V_2,E_2)$, a map $\phi \colon V_1\to V_2$ is a \textit{homomorphism} if $\phi$ preserves edges, i.e. for every $(u,v)\in E_1$ we have $(\phi(u),\phi(v))\in E_2$. To emphasize that $\phi$ is a graph homomorphism
we write $\phi \colon (V_1,E_1)\to (V_2,E_2)$. And to simplify some steps below we use notation $\phi(e)=(\phi(u),\phi(v))$ for an edge $e=(u,v)\in E_1$. This can be extended onto paths
$e_1\ldots e_n$ on $(V_1,E_1)$ by the standard rule $\phi(e_1\ldots e_n)=\phi(e_1)\ldots \phi(e_n)$.

Now we follow notation introduced in \cite{Shi}. A graph homomorphism $\phi$ is \textit{bidirectional} if $(u,v),(u,v')\in E_1$ implies $\phi(v)=\phi(v')$ and
$(w,u),(w',u)\in E_1$ implies $\phi(w)=\phi(w')$. We are ready to define \emph{bd-covers}, i.e. bidirectional maps between edge-surjective graphs.

Now fix a sequence $\G=\seq{\phi_i}_{i=0}^\infty$  of bd-covers $\phi_i \colon (V_{i+1},E_{i+1})\to (V_i,E_i)$, and consider
$$
V_\G=\varprojlim(V_i,\phi_i)=\{ x\in \Pi_{i=0}^\infty V_i : \phi_i(x_{i+1})=x_i \text{ for all }i\geq 0\}
$$
the inverse limit defined by $\G$. As usual, let $\phi_{m,n}=\phi_n\circ \phi_{n+1}\circ \ldots \circ \phi_{m-1}$ and denote the projection from $V_\G$ onto $V_n$ by $\phi_{\infty,n}$.
Denote
$$
E_\G=\{e\in V_\G\times V_\G : e_i\in E_i \text{ for each }i=1,2,\dots \} .
$$
Any $V_i$ is endowed with discrete topology and the space $\mathbb{X}=\prod_{i=0}^\infty V_i$
is endowed with product topology. It is known that this topology is compatible with the metric given by $d(x,y)=0$ when $x=y$ and
$d(x,y)=2^{-k}$
when $x\neq y$ and $k=\min \{i : x_i\neq y_i\}$. In this topology, $V_\G$ is a closed subset of $\mathbb{X}$ and we consider it with topology (and metric)
induced from the space $\mathbb{X}$.

By a \emph{cycle} on graph $G$ we mean any finite sequence of edges starting and ending in the same vertex.
For cycles $c_1,\ldots, c_n$ starting in the same vertex $v$  we denote by $a_1 c_1+\ldots +a_n c_n$
the cycle at $v$ obtained by passing $a_1$ times cycle $c_1$ then $a_2$ times cycle $c_2$, and so on. The length of any path $\eta$ (i.e. the number of edges on it) is denoted $|\eta|$.

Finally, $V(\eta)$ denotes the set of vertexes
on path $\eta$. The following important fact is given in \cite[Lemma 3.5]{Shim2}.
\begin{lem}\label{lem:Tg}
	Let $\G=\seq{\phi_i}$ be a sequence of bd-covers $\phi_i \colon (V_{i+1},E_{i+1})\to (V_i,E_i)$.
	Then $V_\G$ is a zero-dimensional compact metric space and the relation $E_\G$ defines a homeomorphism.
\end{lem}

\section{Odometers with vanishing derivative everywhere}\label{sec:odometers}
In this section we shall show that every odometer can serve as an example of a system with (LRS) property, by showing that each of them can be embedded in $\mathbb{R}$ with vanishing derivative everywhere.
It is worth emphasizing that all systems considered in this section are equicontinuous.

Let $(X,T)$ be a dynamical system.
A point $x\in X$ is \textit{regularly recurrent} if for every open set $U\ni x$ there is $n$ such that $T^{in}(x)\in U$ for every $i=0,1,\ldots$.
If there exists a regularly recurrent point $x\in X$
such that $\overline{\Orb(x,T)}=X$ and additionally $(X,T)$ is equicontinuous, then we say that $(X,T)$ is an \textit{odometer}. A particular example satisfying this definition is any periodic orbit.
There are several equivalent definitions of odometers (see \cite{D05}).

Let $\mathbf{s}=(s_n)_{n\in \N}$ be a nondecreasing
sequence of positive integers such that $s_n$ divides $s_{n+1}$. For each $n\geq 1$ define $\pi_n\colon \Z_{s_{n+1}}\to \Z_{s_n}$ by the natural formula $\pi_n(m)=m\; (\text{mod }s_n)$ and let $G_\mathbf{s}$ denote the following inverse limit
\[
G_\mathbf{s}=\varprojlim_n(\Z_{s_n}, \pi_n)=
\Bigl\{x\in\prod_{i=1}^\infty \Z_{s_n}: x_{n}=\pi_n(x_{n+1})\Bigr \},
\]
where each $\Z_{s_n}$ is given the discrete topology, and on $\prod_{i=1}^\infty \Z_{s_n}$ we have the Tychonoff product topology.
On $G_\mathbf{s}$ we define a natural map $T_{\mathbf{s}}\colon G_\mathbf{s}\to G_\mathbf{s}$ by $T_{\mathbf{s}}(x)_n=x_n+1 \; (\text{mod }s_n)$.
Then $G_\mathbf{s}$ is a compact metrizable space and $T_{\mathbf{s}}$ is a homeomorphism, therefore $(G_{\mathbf{s}},T_{\mathbf{s}})$
is a dynamical system. It is not hard too see that each point in $(G_{\mathbf{s}},T_{\mathbf{s}})$ is regularly recurrent and that $(G_{\mathbf{s}},T_{\mathbf{s}})$ is equicontinuous,
so it is an odometer. On the other hand it is known that every odometer is conjugated to some $(G_{\mathbf{s}},T_{\mathbf{s}})$, e.g. see \cite{D05}.
It is not hard to see that $G_\mathbf{s}$ is infinite when sequence $\mathbf{s}$ is unbounded, and $(G_{\mathbf{s}},T_{\mathbf{s}})$ is a periodic orbit otherwise.

It is clear that every periodic orbit has (LRS) property. It is also clear that $(G_{\mathbf{s}},T_{\mathbf{s}})$ with the standard metric induced by the discrete metric on each $\Z_{s_n}$
is an isometry. We will show that on each $(G_{\mathbf{s}},T_{\mathbf{s}})$ exists an equivalent metric
under which $T_{\mathbf{s}}$ has (LRS). Clearly this statement is nontrivial only when $G_{\mathbf{s}}$ is infinite.

\begin{thm}\label{Th:LRS}
Fix any strictly increasing sequence $\mathbf{s}=(s_n)_{n\in \N}$  of positive integers such that $s_n$ divides $s_{n+1}$. Then there exists
a continuous injective map $\pi \colon G_{\mathbf{s}}\to \R$ such that the map $f$ on the Cantor set $\pi(G_{\mathbf{s}})$ defined by $f=\pi \circ T_{\mathbf{s}}\circ \pi^{-1}$
has derivative $0$.
\end{thm}
\begin{proof}
We may assume that $s_1>1$ and for each $n\geq 1$ denote $k_{n+1}=s_{n+1}/s_n$.
Put $a_1=1/2$ and $b_1=2^{-k_2(s_1-1)}a_1$.

For $i=0,\ldots ,s_1-1$ let $A_i^{(1)}=[i,i+a_1]$. For technical reasons we put $s_0=1$.
Define function $l_1(i)=2^{-k_2[(i-2)(\textrm{ mod }s_1)]}a_1/3$ for $i=0,\ldots, s_1-1$. Note that $\max l_1=l_1(2)=a_1/3$ and $\min l_1=l_1(1)=b_1/3$.
Let $D_i^{(1)}\subset A_i^{(1)}$ be an interval of length $l_1(i)$ placed in the middle of $A_i^{(1)}$, that is $ A_i^{(1)}\setminus D_i^{(1)}$
has two connected components which are intervals of equal length.

Suppose sets $D_i^{(k)}$, $A_i^{(k)}$ are defined for $k=1,\ldots, n$ and $i=0,\ldots, s_k-1$ as well as numbers $a_i,b_i$.
Assume that diameter of each interval satisfies $a_n\geq 3\diam D_i^{(n)}\geq b_n$.
Divide each $D_i^{(n)}$ into $k_{n+1}$ intervals of equal length and disjoint interiors, and enumerate them in such a way that $A_j^{(n+1)}\subset A_i^{(n)}$ provided that $i=j (\textrm{mod }s_n)$.
Note that $b_{n}/k_{n+1}\leq 3\diam A_j^{(n+1)}\leq a_{n}/k_{n+1}$ for every $j$. Denote $a_{n+1}=2^{-n}b_n/k_{n+1}$
and $b_{n+1}=2^{-(n+1)k_{n+2}(s_{n+1}-1)}a_{n+1}$. Let $$l_{n+1}(i)=2^{-(n+1) k_{n+2}[(i-s_n-1) (\textrm{mod }s_{n+1})]}a_{n+1}/3$$ and notice that 
\[
l_{n+1}(s_n+1)>\cdots>l_{n+1}(s_{n+1}-1)>l_{n+1}(s_{n+1})>\cdots>l_{n+1}(s_n).
\] 
Let $D_i^{(n+1)}\subset A_i^{(n+1)}$ be an interval of length $l_{n+1}(i)$ placed in the middle of $A_i^{(n+1)}$. By the construction we have $a_{n+1}\geq 3\diam D_i^{(n+1)}\geq b_{n+1}$.

For any $z\in G_{\mathbf{s}}$ the intersection $\bigcap_{i} D_{z_n}^{(n)}$ is a single point $x_z$, because
$$
\diam D_{z_n}^{(n)}\leq a_{n}\leq 2^{-n+1} b_{n-1}\leq 2^{-n+1} a_1\leq 2^{-n}.
$$
Furthermore is $z,w\in G_{\mathbf{s}}$ and $z_n\neq w_n$ then $x_z\in D_{z_n}^{(n)}\subset \Int A_{z_n}^{(n)}$
and $x_w\in D_{w_n}^{(n)}\subset \Int A_{w_n}^{(n)}$.
This shows that the map $\pi \colon G_{\mathbf{s}} \ni z\mapsto x_z\in \R$ is well defined, continuous and injective.
Denote $X=\pi(G_{\mathbf{s}})$. Then $\pi \colon G_{\mathbf{s}}\to X$ is a homeomorphism and $X$ is a Cantor set.
Define $f=\pi \circ T_{\mathbf{s}}\circ \pi^{-1}$. We are going to show that $f'(x)=0$ for every $x\in X$.

Fix any sequence $x_n\to x$ and let $z=\pi^{-1}(x)$ and $z^{(n)}=\pi^{-1}(x_n)$. We may assume that $x\neq x_n$ for every $n$,
hence there exists a sequence $j_n$ such that $z^{(n)}_{j_n}=z_{j_n}$ and $z^{(n)}_{j_n+1}\neq z_{j_n+1}$.

Observe that there exists at most one $n$ such that $\diam D^{(n)}_{z_n}=b_{n}/3$. Namely, by definition of function $l_n$, such a case happens exactly when $z_n=s_{n-1}$, and if it is the case then $z_i=0<s_{i-1}$ for all $i<n$. Furthermore, $z\in D_{z_n}^{(n)}$, which implies that  $f(z)\in D_{z_n+1 (\textrm{mod }s_{n})}^{(n)}$.
Therefore, if $n$ is sufficiently large then $z_n\neq s_{n-1}$ and so there is $i_n$ such that
$$
\diam D_{z_n}^{(n)}=2^{-nk_{n+1} i_n}a_{n}/3\quad \textrm{ and }\quad \diam D_{f(z)_n}^{(n)}=2^{-nk_{n+1}(i_n+1)}a_{n}/3,
$$
which gives
$$
\frac{\diam D_{f(z)_n}^{(n)}}{\diam D_{z_n}^{(n)}}=2^{-nk_{n+1}}.
$$
Observe that
$$
3\diam A^{(n)}_i \geq \frac{b_{n-1}}{k_{n}}\geq 2^{n-1} a_{n}\geq 2^{n-1}3\diam D^{(n)}_i
$$
hence for $n>3$ we have $\diam A^{(n)}_i/3\geq \diam D^{(n)}_i$.
Therefore, if $p\in  D_{i}^{(n+1)}$ and $q\in D_{j}^{(n+1)}$ for some $i\neq j$, then (for large $n$):
$$
|p-q|\geq \frac{\diam A_{i}^{(n+1)}-\diam D_i^{(n+1)}}{2}\geq \frac{\diam A_i^{(n+1)}}{3}=\frac{\diam D_{i (\text{mod} s_n)}^{(n)}}{3k_{n+1}}.
$$
By the above estimates we obtain that
$$
\frac{|f(z)-f(z^{(n)})|}{|z-z^{(n)}|}\leq \frac{\diam D^{(j_n)}_{f(z)_{j_n}}}{\diam A_{z_{j_n+1}}^{(j_n+1)}/3}\leq\frac{3 k_{j_n+1}\diam D^{(j_n)}_{z_{j_n}+1}}{\diam D_{z_{j_n}}^{(j_n)}}\leq
\frac{3 k_{j_n+1}}{2^{j_n k_{j_n+1}}}\longrightarrow 0.
$$
Indeed $f'(z)=0$ completing the proof.
\end{proof}

We obtain the following immediate corollaries.

\noindent
\begin{thm}\label{thm-A}
Every odometer is conjugate to a homeomorphism $\mathfrak{f}:C\to C$, $C\subseteq \mathbb R$, such that $\mathfrak{f}'\equiv 0$ and $\mathfrak{f}$ extends to a differentiable surjection $\bar{\mathfrak{f}}:\mathbb{R}\to\mathbb{R}$.
\end{thm}
\begin{proof}
By Theorem~\ref{Th:LRS} $(X,T)$ is conjugated to $(K,f)$ with $K\subset \R$ and $f'\equiv 0$. By Theorem~\ref{thm:Jarnik} $f$ extends to a differentiable map $f\colon \R\to \R$
with $f'|_K\equiv 0$. 
\end{proof}
\begin{cor}
For every odometer $(X,T)$ there exists an equivalent metric $\rho$ such that $T$ has (LRS) property with respect to $\rho$.
\end{cor}

\begin{cor}\label{cor-D}
There exists a Cantor set $C\subseteq\mathbb{R}^2$ and a homeomorphism $F$ such that $F$ has the (LRS) property, $C=\bigcup_{i\in I} M_i$ where $I$ is uncountable, $M_i\cap M_j=\emptyset$ for $i\neq j$
and $(M_i,F)$ if minimal for every $i$.
\end{cor}
\begin{proof}
	Let $(C,f)$ be any odometer provided by Theorem~\ref{Th:LRS}. Let $F=f\times f$, where $C\times C$ is endowed with metric $\rho((x,y),(p,q))=d(x,p)+d(y,q)$.
	Take any $(p,q)\in B((x,y),\eps)$ where $\eps=\min\{\eps_x,\eps_y\}$ and assume that $(p,q)\neq (x,y)$. We may assume without loss of generality that $x\neq p$
	Note that
	\begin{eqnarray*}
		\rho(F(x,y),F(p,q))&=&\rho((f(x),f(y),(f(p),f(q)))\\
		&=&d(f(x),f(p))+d(f(y),f(q))\leq d(f(x),f(p))+d(y,q)\\
		&<& d(x,p)+d(y,q)=\rho((x,y),(p,q)).
	\end{eqnarray*}
	This shows that $F$ has (LRS) and clearly $C\times C$ is homeomorphic to $C$.
	
	But if we fix any $y\in C$ then for $p\neq q$ the pair $(x,p)$ defines a minimal set different than $(x,q)$.
\end{proof}

\section{(LRS) without equicontinuity}

As we could see in Section~\ref{sec:odometers}, in every odometer we can replace metric to an equivalent one in such a way that (LRS) property is satisfied. Intuitively it seems that this property is connected with equicontinuity, i.e. distance between orbits cannot increase because of shrinking. But this is somehow misleading, because we have only local shrinking and we cannot completely control it during the evolution of orbits. To show this phenomenon precisely, we will construct two examples which are not odometers. First one will be minimal weakly mixing system (hence sensitive) which has (LRS) property. Second one will be transitive
but not minimal.
\vspace{0.025cm}

\noindent
\begin{thm}\label{thm-B}
There exists a minimal weakly mixing Cantor set homeomorphism $T:X\to X$ that embeds in $\mathbb{R}$ with vanishing derivative everywhere.
\end{thm}

\begin{proof}
For every integer $n>0$ we define a special vertex $v_{n,0}\in V_n$ and $V_0=\{v_{0,0}, v_{0,1,1},v_{0,2,1}, v_{0,2,2}\}$. We use $V_0$ to define two cycles
	$$
	c_{0,1} : v_{0,0}\to v_{0,1,1}\to v_{0,0}\qquad\text{ and }\quad c_{0,2} : v_{0,0}\to v_{0,2,1}\to v_{0,2,2} \to v_{0,0}
	$$
	we put $G_0=(V_0,E_0)$ where $E_0$ is the set of edges defined by cycles $c_{0,1}$, $c_{0,2}$.
	Next we will specify other vertexes in $V_n$ and accompanying edges, so that a graph $G_n$ is defined.
	Our aim is to construct a special sequence of bd-covers. In particular we put $\phi_n(v_{n+1,0})=v_{n,0}$.
	We embed in each $V_n$ exactly $2$ additional cycles $c_{n,1}, c_{n,2}$ (of appropriate length, which will be clear from the context), such that each cycle starts and ends in $v_{n,0}$
	and all the other vertexes are pairwise distinct. For $i=1,2$ we put
	$$
	\phi_n(c_{n+1,i})=2c_{n,1}+c_{n,i}+c_{n,2}.
	$$
It is clear that both $\phi_n(c_{n+1,1})$, $\phi_n(c_{n+1,2})$ start and terminate in the same vertex of $V_n$ so property form definition of bd-cover is preserved.
By the definition we also have that $|c_{n,2}|=|c_{n,1}|+1$ for every $n$. Then for every $n\geq 1$ there exists $k_n\geq $ such that
$$
V_n=\{v_{n,0},v_{n,1,1},\ldots, v_{n,1,k_n},v_{n,2,1},\ldots,v_{n,2,k_n+1}\}.
$$

Let $\G=\seq{\phi_i}_{i=0}^\infty$ be the sequence defined above and 
we denote by $T_\G\colon V_\G\to V_\G$ the homeomorphism induced by $E_\G$
in view of Lemma~\ref{lem:Tg}. It is clear that image of $\phi_n$ on each cycle in $G_n$ covers whole $V_n$, which shows that $(V_\G, T_\G)$ is minimal (e.g. see \cite{Shi}).
By the definition of cycles $c_{n+1,1}, c_{n+1,2}$ we see that for every $n$ there are paths joining cycle $c_{n,1}$ with itself fo length $m$ and $m+1$ for some $m$.
This immediately implies that $(V_\G, T_\G)$ is weakly mixing.

It remains to define metric on $V_\G$ that will give the vanishing derivative. We will proceed in a way similar to the proof of Theorem~\ref{thm-A}.

For technical reasons we put $s_{-1}=1$ and
$s_n=|V_n|$ for $n=0,1,\ldots$.
Put $a_0=1/2$ and $b_0=2^{-2s^2_0}a_0/3$.

For each $n>0$ and $w\in V_n$ we define the function $\psi_n\colon V_n\to (0,1)$ by putting
$$
\psi_n(w)=
\begin{cases}
2^{-2s_n^2}a_{n-1}/3,&\text{if }w = v_{n,0},\\
2^{-2s_n^2-is_n}a_{n-1}/3,&\text{if }w = v_{n,e,i},\ 0<i \leq |c_{n-1,1}|,\ e\in \{1,2\},\\
2^{-s_n^2-is_n}a_{n-1}/3,&\text{if }w = v_{n,e,i},\ i > |c_{n-1,1}|,\ e\in \{1,2\}.\\
\end{cases}
$$
For $i=-s_0+1,\ldots, s_0-1$ let $A_i^{(0)}=[i,i+a_0]$ and define
$$
l_n(i)=
\begin{cases}
\psi _n (v_{v_{n,0}}),&\text{if }i=0,\\
\psi _n (v_{v_{n,1,i}}),&\text{if }0<i,\\
\psi _n (v_{v_{n,2,-i}}),&\text{if }i<0,\\
b_n & \textit{otherwise}.
\end{cases}
$$

For $n\geq 0$ let $D_n^{(1)}\subset A_n^{(1)}$ be an interval of length $l_n(i)$ placed in the middle of $A_i^{(n)}$, that is $ A_i^{(n)}\setminus D_i^{(n)}$
has two connected components which are intervals of equal length.

Suppose sets $D_i^{(k)}$, $A_i^{(k)}$ are defined for $k=1,\ldots, n$ and $i=-s_k+1,\ldots, s_k-1$.
We put $a_{n}=\max_{i} \diam A_i^{(n)}$ and $b_{n}=2^{-s_{n}^2}a_{n}$.

Divide each $D_i^{(n)}$ into $12$ intervals of equal length and disjoint interiors.
For each vertex $w=v_{n+1,j,i}$ we assign one interval in $D_{r}^{(n)}$ if $\phi_n(w)=v_{n,1,r}$ and in $D_{-r}^{(n)}$ if $\phi_n(w)=v_{n,2,r}$. Since for every $v\in V_n$ we have $|\phi_n^{-1}(v)|\leq 12$
we can assign pairwise disjoint intervals to different vertexes. Finally,  name intervals assigned to $v_{n+1,1,i}$ as $A^{(n+1)}_i$ and interval assigned to $v_{n+1,2,i}$ as $A^{(n+1)}_{-i}$.
By $A^{(n+1)}_0$ we denote interval corresponding to $v_{n,0}$.

For any $z\in V_\G$ let $\eta(z)_n=r$ if $z_n=v_{n,1,r}$; $\eta(z)_n=-r$ if $z_n=v_{n,2,r}$ and $\eta(z)_n=0$ if $z_n=v_{n,0}$. Then the intersection $\bigcap_{i} D_{\eta(z)_n}^{(n)}$ is a single point $x_z$, because
$$
\diam D_{z_n}^{(n)}\leq 2^{-s_n}a_{n}\leq 2^{-n}.
$$
Furthermore is $z,w\in V_\G$ and $z_n\neq w_n$ then $x_z\in D_{\eta(z)_n}^{(n)}\subset \Int A_{\eta(z)_n}^{(n)}$
and $x_w\in D_{\eta(w)_n}^{(n)}\subset \Int A_{\eta(w)_n}^{(n)}$.
This shows that the map $\pi \colon V_\G \ni z\mapsto x_z\in \R$ is well defined, continuous and injective.
Denote $X=\pi(V_\G)$. Then $\pi \colon V_\G \to X$ is a homeomorphism and $X$ is a Cantor set.
Define $f=\pi \circ T_\G\circ \pi^{-1}$. We are going to show that $f'(x)=0$ for every $x\in X$.

Fix any sequence $x_n\to x$ and let $z=\pi^{-1}(x)$ and $z^{(n)}=\pi^{-1}(x_n)$. We may assume that $x\neq x_n$ for every $n$,
hence there exists a sequence $j_n$ such that $z^{(n)}_{j_n}=z_{j_n}$ and $z^{(n)}_{j_n+1}\neq z_{j_n+1}$.

Observe that by the definition of function $\phi_n$ there exists at most one $n$ such that $z_n=v_{n,j,|c_{n-1,1}|}$, because $\phi_k(v_{k,0})=v_{k-1,0}$
for every $k$ and $\phi_{n-1}(v_{n,j,|c_{n-1,1}|})=v_{n-1,0}$. Therefore, if $n$ is sufficiently large then there are $r_n$ and $t_n\geq r_n+1$ such that
\begin{eqnarray*}
\diam D_{\eta(z)_n}^{(n)}&=&2^{-r_n s_n}a_{n-1}/3\\
\diam D_{\eta(T_\G(z))_n}^{(n)}&=&2^{-t_ns_n}a_{n-1}/3\leq 2^{-s_nr_n-s_n}a_{n-1}/3.
\end{eqnarray*}
and if $p\in  D_{r}^{(n)}$ and $q\in D_{s}^{(n)}$ and $r\neq s$, then $d(p,q)\geq \diam A_{r}^{(n)}/3.$ Moreover,
$$
|p-q| \geq \frac{\diam A^{(n+1)}_r - \diam D^{(n+1)}_r} {2} > \diam D^{(n+1)}_r,
$$
because, by the construction above, we obtain
$$
\diam A^{(n+1)}_i = \frac {\diam D^{(n)}_r}{12} \geq \frac{b_n}{12} = \frac{a_n}{12 \cdot 2^{3s_{n}^2}} >\frac{a_n}{12 \cdot 2^{3s_{n+1}^2}} \geq \diam D^{(n+1)}_j.
$$

By the above estimates we obtain that
$$
\frac{|f(z)-f(z_n)|}{|z-z_n|}\leq \frac{\diam D^{(j_n)}_{\eta(T_\G(z))_{j_n}}}{\diam A_{\eta(z)_{j_n+1}}^{(j_n+1)}/3}\leq\frac{\diam D^{(j_n)}_{\eta(T_\G(z))_{j_n}}}{\diam D_{\eta(z)_{j_n}}^{(j_n)}/12}\leq 12 \cdot 2^{-s_n}\longrightarrow 0.
$$
Indeed $f'(z)=0$ completing the proof.
\end{proof}

\noindent
\begin{thm}\label{thm-C}
	There exists a transitive, nonminimal and periodic point free Cantor set homeomorphism that embeds in $\mathbb{R}$ with vanishing derivative everywhere.
\end{thm}
\begin{proof}
The calculations and main features of the construction are exactly the same as in the proof of Theorem~\ref{thm-B}.
In each step we use two cycles, but now covering relation is different.
We define
$$
\phi_n(c_{n+1,1})=3c_{n,1}, \quad \phi_n(c_{n+1,2})=2c_{n,1}+2c_{n,2}+c_{n,1}.
$$
This system is not minimal, since inverse limit of cycles $c_{n,1}$ defines an odometer. But it is transitive, because cycle $\phi_n(c_{n+1,2})=V_n$ and $\phi_n(c_{n+1,2})$ covers two copies of $c_{n,2}$.
\end{proof}

\section{Attractor-repellor pair}
In the following example we consider (LRS) property instead of vanishing derivative, since the system described by us contains isolated points where the derivative is undefined.
\begin{thm}\label{Th:attr:rep}
Every minimal dynamical system $(X,T)$ with (LRS) property can be extended to a non-transitive dynamical system $(Z,F)$
with (LRS) property.
\end{thm}

\begin{proof}
Let $(X,T)$ be a minimal dynamical system with (LRS) property. Fix any $z\in X$.
There exists a nested sequence of closed-open neighborhoods $U_n$ of $z$, such that $d(T(z),T(x))<d(z,x)$ for every $x\in U_1$. Going to a subsequence (removing some of the sets $U_n$) if necessary we can find
an increasing sequence $k_n$ such that $f^{-k_n}(z)\in U_n$ and $f^{-i}(z)\not\in U_n$ for $0< i < k_n$. Observe that set $U_{n}\setminus U_{n+1}$ is closed
for every $n$, and so the following number
$$
0<a_n < \min_{y\in U_n\setminus U_{n+1}}d(z,y)-d(T(z),T(y))
$$
exists. We may assume that $a_{n+1}<a_n/2$ and $a_1<1/2$. For $l\geq -k_1$ put
$y_{l}=(T^{l}(z)),1-2^{-l+k_1})$.
For each $n\geq 1$ we set $y_{-k_n}=(T^{-k_n}(z),-1+a_{n+1}/2))$, and for $y_{-k_n+1}=(T^{-k_n+1}(z),-1+a_{n+1}/2+a_{n}))$.
Finally, for $j=-k_n+2,\ldots, -k_{n-1}$ we put $y_{j}=(T^{j}(z),-1+\frac{(-j-k_{n-1})}{2(k_{n}-k_{n-1}-1)}(a_n+a_{n+1})+a_{n}/2))$.
Finally, we put $Z=X\times \{-1,1\}\cup \bigcup_{j\in \Z}\{y_j\}$. We endow $Z$ with metric $\rho((p,q),(a,b))=d(p,a)+d(q,b)$.

Note that $\lim_{j\to \infty}\pi_2(y_j)=1$ and $\lim_{j\to -\infty}y_j=-1$.
Furthermore map $F\colon Z\to Z$ defined by $F(y_j)=y_{j+1}$ and $F(p,\pm 1)=(T(p),\pm 1)$ is continuous.

Note that every point $y_j$ is isolated in $Z$ so $F$ has (LRS) at all of these points. It also has (LRS) at points of $X\times \{1\}$
because $T$ has (LRS) on $X$ and points $y_j$ are attracted by the set $X\times \{1\}$.
We have $\lim_{n\to \infty} y_{-k_n}=(z,-1)$ so for every $y\in X\setminus \{z\}$ there exists an open set $V_y\subset Z$, with $\operatorname{diam}(V_y)<\operatorname{diam}(U_1)$,
such that $(y,-1)\in V$ and $y_{-k_n}\not\in V_y$ for every $n$. But for every $j\not\in \{-k_n : n\}$, $j<-k_1$ we have
$$
\dist (y_j,X\times\{-1\})> \dist (F(y_j),X\times\{-1\})=\dist (y_{j+1},X\times\{-1\}).
$$
This shows (LRS) property of $F$ at any point in $Z$ other than $(z,-1)$. Consider open set $V\subset U_1\times [-1,-1+a_1/2)$
with $\diam V<\eps_z$. Fix any $y\in V$. If $y\in X\times \{-1\}$ then by (LRS) of $T$ we have $\rho((z,-1),y)>\rho((T(z),-1),F(y))$ provided that $y\neq (z,-1)$.
If $y=y_j$ for some $j$ then the only possibility for $\rho((z,-1),y)\leq \rho((T(z),-1),F(y))$ to occur is when $y=y_{-k_n}$.
But then
\begin{eqnarray*}
\rho(F(z,-1),F(y_{-k_n}))&=& \rho((T(z),-1),(T^{k_n+1}(z),-1+a_{n+1}/2+a_{n}))\\
&=&d(T(z),T^{k_n+1}(z))+a_{n+1}/2+a_{n}\\
&<& d(z,T^{k_n}(z))+a_{n+1}/2=\rho((z,-1),(T^{k_n}(z),-1+a_{n+1}/2))\\
&=& \rho((z,-1), y_{k_n}).
\end{eqnarray*}
The proof of (LRS) of $F$ is completed.
\end{proof}

\begin{rem}
By the construction in Theorem~\ref{Th:attr:rep}, $Z\subseteq X\times [-1,1]$ but we can assume $Z\subseteq X\times J$, where $J$ is an arbitrary interval. Moreover, $X\times \{1\}$ is an attractor and we can assume monotonicity in the second coordinate near this attractor, i.e. there exists a nondegenerate interval $L = [c,1]$ such that, for any point $(x,y)\in X\times L$ we have $y < \pi_2 (F(x,y))$. Clearly, we can also perform the construction in such a way that there exists $L^\prime = [c^\prime, 1]$ such that
\begin{equation}\label{EQ:attr}
  |1-y| > 3|1-\pi_2 (F (x,y))|
\end{equation}
for each $(x,y)\in X\times L^\prime$.
\end{rem}

Apart from the attractor and repellor, the system constructed in Theorem \ref{Th:attr:rep} consists only of isolated points. The following result shows that we can construct a system with (LRS) having a fixed point and having no isolated points.

\begin{thm}
	\label{thm-E} There exists a Cantor set $W\subseteq\mathbb{R}^2$ and a nontransitive homeomorphism $G$ with the (LRS) property such that the set of periodic points of $G$ consists of a single fixed point.
\end{thm}

\begin{proof}
Let $(X_1, T_1)$ be a Cantor system having (LRS). Then by Theorem \ref{Th:attr:rep}, this system can be extended to a system $(Z,F)$ with (LRS) and we can assume that $Z\subseteq X_1\times [-2,0]$. Note that we can also assume $X_1 \subseteq [0,1]$, $X_1 \times \{-2\}$ is the repellor of $(Z,F)$, $X_1\times \{0\}$ is the atractor of $(Z,F)$ and $|y|<|\pi_2 (F(y))|$ for $y\in [-1,0)$.
We take one more system $(X_2,T_2)$, $X_2\subseteq [0,1]$ having (LRS). Then the product system $(\bar W, \bar G) :=(Z\times X_2, F\times T_2)$  has (LRS) (see Lemma \ref{Th:attr:rep}) with respect to metric $ \bar d(w_1, w_2) = |x_1-x_2| + |y_1-y_2| + |z_1-z_2|$ where $w_1,w_2 \in X_1 \times [-2,-1] \times X_2$, $w_i = (x_i,y_i,z_i)$.

We construct a new system $(W,G)$ with the help of $(\bar W, \bar G)$. The new system $(W,G)$ coincides with $(\bar W, \bar G)$ on $X_1 \times [-2,-1] \times X_2$. Thus the system $(W,G)$ has (LRS) on a subset of $X_1 \times [-2,-1] \times X_2$.

We define a new metric $d$ on $X_1 \times [-1,0] \times X_2$ such that the attractor $X_1 \times \{0\} \times X_2$ shrinks to a singleton (in fact, to a fixed point), and  $(W,G)$ defined with the help of $(\bar W, \bar G)$ has (LRS) on $X_1 \times [-1,0] \times X_2$. For $w_1,w_2\in X_1 \times [-1,0] \times X_2$, $w_i= (x_i,y_i,z_i)$, let
$d(w_1,w_2) = |x_1y_1 - x_2y_2| + |y_1 -y_2| + |z_1y_1 - z_2 y_2|$. Clearly, $d = \bar d$ on $X_1 \times \{\-1\}\times X_2$ and $X_1 \times \{0\} \times X_2$ become a singleton. It remains to show that $G$ has (LRS) on $X_1 \times [-1,0]$, i.e. in a sufficiently small neighborhood $U$ of every point $w_1 \in X_1 \times [-1,0] \times X_2$, $w_1= (x_1,y_1, z_1)$ we have
$$
|x_1y_1 - x_2y_2| + |y_1 - y_2| + |z_1y_1 - z_2 y_2| >$$
\begin{equation}\label{EQ:New:Metric}
> |T_1 (x_1)\pi_2 (F(y_1 )) - T_1 (x_2)\pi_2 (F(y_2  ))| + |\pi_2 (F(y_1 )) - \pi_2 (F(y_2))| +
\end{equation}
$$+ |T_2 (z_1)\pi_2 (F( y_1)) - T_2 (z_2) \pi_2 (F(y_2 ))|
$$
whenever $w_2 \in U$ (We write $\pi_2 (F( y))$ instead of $\pi_2 (F( x, y))$.).

Recall that we can assume $|y| > |\pi_2 (F(y))|$ for all $y\in [-1,0)$. Thus because $T_1, T_2$ have (LRS) sufficiently close to $(x_1,z_1)$
\begin{eqnarray*}
  |y_1| (|x_1 - x_2| + |z_1 - z_2| )&\geq& |y_1|\cdot (|T_1(x_1) - T_1 (x_2)| + |T_2(z_1 ) - T_2 (z_2)| ) \\
   &>&  |\pi_2 (F(y_1))| \cdot (|T_1(x_1) - T_1 (x_2)| + |T_2(z_1 ) - T_2 (z_2)| ).
\end{eqnarray*}
This proves (\ref{EQ:New:Metric}) since $F$ has (LRS), i.e.
$$
|y_1|\cdot (|x_1 - x_2| + || y_1 - y_2) > |\pi_2 (F (y_1))|\cdot (|T_1 (x_1) - T_1 (x_2)| + |\pi_2 (F(1))- \pi_2 (F(y_2))|)
$$
and $(x_1, y_1)$ an isolated point in $Z$.

It remains to prove (LRS) at $w_1 = (x_1, 0, z_1)$. Since in this case we have $y_1 = \pi_2 (F(x_1)) = 0$, it remains to prove
$$
|x_2y_2| + |y_2| + |z_2y_2| > |T_1 (x_2) \pi_2 (F(y_2))| + |\pi_2 (F(y_2))| + |T_2(z_2)\pi_2 (F(y_2))|,
$$
which is true because we can assume that $|y_1| >3 |\pi_2 (F(y_1)) |$ (see (\ref{EQ:attr})) close to $X_1 \times \{0\}\times X_2$. This finishes the proof.
\end{proof}

\begin{que}
Is every minimal Cantor set homeomorphism conjugate to a homeomorphism with vanishing derivative everywhere?
\end{que}
\section*{Acknowledgements}
The authors are grateful to K. Ciesielski and J. Jasinski for valuable feedback and comments at the Spring Topology and Dynamics Conference, New Jersey City University, March 8-11, 2017. Ciesielski and Jasinski's related, most recent work can be found in \cite{Ciesielski2}, and in \cite{Ciesielski3} where the notion \textit{locally radially shrinking} is called \textit{pointwise shrinking} instead. J. Kupka was supported by the NPU II project LQ1602 IT4Innovations excellence in science, and the collaboration with him was made possible by MSK grant 01211/2016/RRC ``Strengthening international cooperation in science, research and education''. J. Boro\'nski's work was supported by National Science Centre, Poland (NCN), grant no. 2015/19/D/ST1/01184. Research of P. Oprocha was supported by National Science Centre, Poland (NCN), grant no. 2015/17/B/ST1/01259.

\end{document}